\documentclass[11pt,twoside]{amsart}

\usepackage{amsmath}
\usepackage{amsthm}
\usepackage{amsfonts, amssymb}
\usepackage{mathrsfs}
\usepackage[all,line]{xy}
\usepackage{url}
\usepackage{graphics}

\usepackage{latexsym}
\usepackage[dvips]{graphicx}

\theoremstyle{plain}
\newtheorem{lemma}{Lemma}[section]
\newtheorem{proposition}[lemma]{Proposition}

\newtheorem{corollary}[lemma]{Corollary}
\theoremstyle{remark}

\newtheorem{remark}[lemma]{Remark}

\theoremstyle{definition}
\newtheorem{definition}[lemma]{Definition}
\newtheorem{example}[lemma]{Example}

\def\ce{\mbox{$\searrow \! \! \! \! \! ^e \: \: $}}
\def\ee{\mbox{$\nearrow \! \! \! \! \! \! \! ^e \ \ $}}

\def\k{\mathcal{K}}
\def\x{\mathcal{X}}

\def\L{\mathcal{L}}
\newcommand{\scp}{\mbox{$\ \searrow \! \! \! \! \searrow \! \! \! \! \ \ \ $}}

\newcommand{\esc}{\mbox{$\ \searrow \! \! \! \! \searrow \! \! \! \! ^e \ $}}
\begin{document}

\title{A note on the homotopy type of the Alexander dual}
\author[E.G. Minian]{El\'\i as Gabriel Minian}
\author[J.T. Rodr\'iguez]{Jorge Tom\'as Rodr\'iguez}

\address{Departamento  de Matem\'atica--IMAS\\
 FCEyN, Universidad de Buenos Aires\\ Buenos
Aires, Argentina.}

\email{gminian@dm.uba.ar} 
\email{jtrodrig@dm.uba.ar}

\begin{abstract}
We investigate the homotopy type of the Alexander dual of a simplicial complex. In general the homotopy type of $K$ does not determine the homotopy type of its dual $K^*$. Moreover, one can construct for each finitely presented group $G$, a simply connected simplicial complex $K$ such that $\pi_1(K^*)=G$. We study sufficient conditions on $K$ for $K^*$ to have the homotopy type of a sphere. We also extend the simplicial Alexander duality to the context of reduced lattices.
\end{abstract}

\subjclass[2000]{55U05, 55P15, 57Q05, 57Q10, 55M05, 06A06.}

\keywords{Alexander duality, simplicial complexes, lattices, homotopy types, finite topological spaces.}

\maketitle

\section{Introduction}
Let $A$ be a compact and locally contractible proper subspace of $S^n$. The classical Alexander duality theorem asserts that the reduced homology groups $H_i(S^n-A)$  are isomorphic to the reduced cohomology groups $H^{n-i-1}(A)$ (see for example \cite[Thm. 3.44]{ha}). The combinatorial (or simplicial) Alexander duality is a special case of the classical duality: if $K$ is a finite simplicial complex and $K^*$ is the Alexander dual with respect to a ground set $V\supseteq K^0$, then for any $i$
 
$$H_i(K)\cong H^{n-i-3}(K^*).$$

Here $K^0$ denotes the set of vertices (i.e. $0$-simplices) of $K$ and $n$ is the size of $V$. A very nice and simple proof of the combinatorial Alexander duality can be found in \cite{bt}. An alternative proof of this combinatorial duality can be found in \cite{ba}.

In these notes we relate the homotopy type of $K$ with that of $K^*$. We show first that, even though the homology of $K$ determines the homology of $K^*$ (and vice versa), the homotopy type of $K$ does not determine the homotopy type of $K^*$. Moreover, for any finitely presented group $G$, one can find a simply connected complex $K$ such that its Alexander dual, with respect to some ground set $V$, has fundamental group isomorphic to $G$. In the same direction, we exhibit an example of a complex with the homotopy type of a sphere whose dual is not homotopy equivalent to a sphere. However we prove that if $K$ simplicially collapses to the boundary of a simplex, then $K^*$ is homotopy equivalent to a sphere. To prove this result we use the nerve of the dual. We also use the nerve to find an easy-to-check sufficient condition for a complex to simplicially collapse to the boundary of a simplex.

In the last section of these notes we extend the duality to the context of reduced lattices. A reduced lattice is a finite poset with the property that any subset which is bounded below has an infimum. Any finite simplicial complex can be seen as a reduced lattice by means of its face poset. We define the Alexander dual for reduced lattices and show that the duality theorem remains valid in this context. When the poset is the face poset of a simplicial complex, the construction coincides with the simplicial one. 

\section{Preliminaries}

Let $K$ be finite simplicial complex and let $V$ be a set which contains the set $K^0$ of $0$-simplices of $K$. The Alexander dual of $K$ (with respect to the fixed set $V$) is the simplicial complex
$$K^*=\{\sigma\subset V,\ \sigma^c\notin K\}.$$
Here $\sigma^c=V\backslash\ \sigma$, the complement of $\sigma$ in $V$. It is clear that $K^{**}=K.$ Note that the set $V$ is implicit in the definition of the dual.

The simplicial Alexander dual $K^*$ allows us to investigate the homology of $K$ but in general the homotopy type of $K^*$ does not determine the homotopy type of $K$. Moreover, the fundamental group of $K^*$ does not provide information about the fundamental group of $K$. In fact, one can prove the following.

\begin{proposition}\label{piuno}
 For any given finitely presented group $G$, there exists a connected compact simplicial complex $K$ such that $\pi_1(K)=G$ and such that its Alexander dual $K^*$ with respect to any $V\supseteq K^0$ is simply connected.
\end{proposition}

\begin{proof}
Since $G$ is finitely presented, there exists a connected $2$-dimensional finite simplicial complex $K$ such that $\pi_1(K)=G$. We can suppose without loss of generality that $K$ has more than six vertices.
The dual of $K$, with respect to any $V\supseteq K^0$ contains the whole $2$-skeleton of the simplex spanned by $V$, since the complement of any subset of three elements of $V$ is not a simplex in $K$, by a cardinality argument. It follows that $K^*$ is simply connected.
\end{proof}

\begin{corollary}
For any finitely presented group $G$ there is a simply connected complex whose dual, with respect to some $V$, has fundamental group isomorphic to $G$.
\end{corollary}

In the same direction, the following example shows two homotopy equivalent simplicial complexes $K, L$ such that $K^0=L^0=V$ and such that their duals $K^*,L^*$ (with respect to $V$) are not homotopy equivalent.

\begin{example} 
Let $M$ be a triangulation of the Poincar\'e homology $3$-sphere and let $S$ be the boundary of a $4$-simplex whose vertices are contained in the set $V=M^0$. Similarly as in the proof of Proposition \ref{piuno},  since any triangulation $M$ of the homology $3$-sphere has more than $7$ vertices and $M$ and $S$ are $3$-dimensional, their duals $K=M^*$ and $L=S^*$ (with respect to $V$) are simply connected. Since $K$ and $L$ have the homology of a sphere $S^r$, it follows that they are in fact homotopy equivalent. Moreover, $K^0=L^0=V$ and their duals are respectively $M$ and $S$, which are not homotopy equivalent. 
\end{example}

In particular, the last example shows that the dual of a complex which is homotopy equivalent to a sphere need not be  homotopy equivalent to a sphere. The next lemma shows that, when we restrict ourselves to simplicial collapses, the duals preserve the homotopy type. We refer the reader to \cite{co} for the basic notions on simplicial collapses and expansions and simple homotopy types. As usual, we will denote an elementary  simplicial collapse by $K\ce L$ and, in general,  $K\searrow L$ will denote a simplicial collapse.

\begin{lemma} Let $L$ be a subcomplex of $K$ and let $V$ be a set containing $K^0$. Then $K\ce L$ if and only if $K^*\ee L^*$.
Consequently, if $K\searrow L$, then $K^*\nearrow L^*$.
\end{lemma}

 \begin{proof}
 Note that if $L=K\backslash\{\tau,\sigma\}$ with $\tau$ a free face of $\sigma$, then $K^*=L^*\backslash\{\sigma^c,\tau^c\}$ with $\sigma^c$ a free face of $\tau^c$.
 \end{proof}
 

Recall that the nerve $N(K)$ of a simplicial complex $K$ is the complex whose vertices are the maximal simplices (=facets) of $K$ and the simplices are the subsets of facets with non-empty intersection. It is well-known that $N(K)$ is homotopy equivalent to $K$.

\begin{lemma} Let $\dot\tau$ be the boundary of a simplex and let $V$ be a set such that $\tau^0\subsetneq V$. Then $(\dot\tau)^*$ is homotopy equivalent to the sphere $S^{n-1}$, where $n=\# V-\# \tau^0$. 
\end{lemma}

\begin{proof}
If $n=1$, $V=\tau^0\cup\{v\}$ and $(\dot\tau)^*$ is the disjoint union of the simplex $\tau$ and the vertex $v$. Then $(\dot\tau)^*$ is homotopy equivalent to $S^0$.

In general, if $V=\tau^0\cup\{v_1,\ldots,v_n\}$, $(\dot\tau)^*$ has $n+1$ maximal simplices, namely the simplices $\eta_i$ with vertex sets $\tau^0\cup\{v_1,\ldots,\hat{v_i},\ldots,v_n\}$, for $i=1,\ldots,n$, and $\eta_{n+1}$ whose vertex set is $\{v_1,\ldots,v_n\}$. The intersection of all these simplices is empty but any other intersection is non-empty. Then the nerve of $(\dot\tau)^*$ is the boundary of the $n$-simplex and therefore $(\dot\tau)^*$ is homotopy equivalent to $S^{n-1}$.
\end{proof}

\begin{corollary} 
If $K$ collapses to the boundary of a simplex, then $K^*$ is homotopy equivalent to a sphere.
\end{corollary}

We can use the nerve of the complex to find an easy-to-check sufficient condition for a complex to collapse to the boundary of a simplex. Note that in many cases, the nerve of a complex $K$  is much smaller than $K$. Moreover, in \cite{bm} it is proved that any complex $K$ \it strong collapses \rm to the square-nerve $N^2(K)=N(N(K))$. In particular, $K \searrow N^2(K)$. The strong collapses are easier to handle than the usual collapses. The concrete definition is the following.

\begin{definition}
Let $K$ be a complex and let $v\in K$ be a vertex. We denote by $K\smallsetminus v$ the full subcomplex of $K$ spanned by the vertices different from $v$ (the \textit{deletion} of the vertex $v$). We say that there is an \textit{elementary strong collapse} from $K$ to $K\smallsetminus v$ if the link of the vertex $lk(v, K)$ is a simplicial cone $v'L$, for some $v'$. In this case we say that $v$ is \textit{dominated} (by $v'$) and we denote $K \esc K\smallsetminus v$. There is a \textit{strong collapse} from a complex $K$ to a subcomplex $L$ if there exists a sequence of elementary strong collapses that starts in $K$ and ends in $L$. In this case we write $K \scp L$.
\end{definition}


\begin{figure}[h]
\centering
\includegraphics[width=6.00in,height=1.00in]{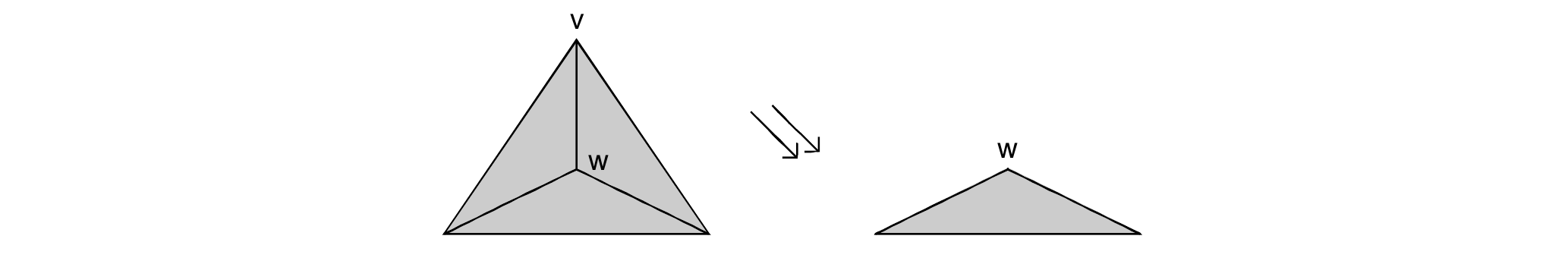}
\caption{An elementary strong collapse.}
\label{strong}
\end{figure}

It is easy to see that $K\scp L$ implies $K\searrow L$. We refer the reader to \cite{bm} for a comprehensive exposition on strong collapsibility and its relationship with simplicial collapsibility. The following lemma shows that this kind of collapses behaves well with respect to the nerve construction. 

\begin{lemma}\label{nerviocolapso}
If $K \scp L$, then $N(K)\scp N(L)$.
\end{lemma}

\begin{proof}
We may suppose that $K\esc L$, i.e. $L= K\setminus \{v\}$ with the vertex $v$  dominated by $w$.
Consider the simplicial map $f:N(L)\rightarrow N(K)$ defined in the vertices of $N(L)$ by 

\[
f(\sigma)= \left\{ \begin{array}{lcl}
                  \sigma & \mbox{ if } & \sigma \in N(K) \\
                    &             &     \\
                  v\sigma & \mbox{ if } & \sigma \notin N(K)
\end{array}
\right.
\]

It is easy to see that $\bigcap \sigma_i \neq \phi$ if and only if $\bigcap f(\sigma_i) \neq \phi$. Therefore we only need to prove that $N(K) \scp f(N(L))$. By \cite[Lemma 3.3]{bm}, it suffices to check that every vertex $\gamma\in N(K)\setminus f(N(L)$ is dominated by a vertex of $f(N(L))$.

Let $\gamma$ be a vertex in $N(K)\setminus f(N(L))$. Since $\gamma\notin f(N(L))$, then $\gamma=v\gamma'$ with $\gamma'$ not maximal in $L$. Therefore there exists $\tau \in L$ a maximal simplex with $\gamma'\subsetneq \tau$. We will show that $\gamma$ is dominated by $\tau$ in $N(K)$.

 Let $\{\sigma_0,...,\sigma_l\}\in lk(\gamma,N(K))$ (i.e. $\cap\sigma_i \bigcap \gamma\neq \phi$). We need to prove that $\cap\sigma_i \bigcap \tau\neq \phi$. If $v\in \cap\sigma_i \bigcap \gamma$, then $v\in \sigma_i$. Since $w$ dominates $v$ and $\sigma_i$ is maximal in $K$, we conclude that $w\in \sigma_i$ and therefore $w\in \cap\sigma_i \bigcap \tau$. If $v\notin \cap\sigma_i \bigcap \gamma$, then $\cap \sigma_i \bigcap \gamma \subseteq \gamma'$. Since $\gamma'\subsetneq \tau$, it follows that $\cap\sigma_i \bigcap \tau\neq \phi$

\end{proof}

Note that in general the previous lemma is not true for simplicial collapses.

\begin{corollary}
Let $K$ be a simplicial complex such that $N(K)\scp\dot\sigma$, where $\dot\sigma$ is the boundary of a simplex. Then $K\searrow \dot\sigma$, and therefore $K^*$ is homotopy equivalent to a sphere.
\end{corollary}

\begin{proof}
By Lemma \ref{nerviocolapso}, $N(N(K))\scp N(\dot\sigma)=\dot \sigma$ and by \cite[Proposition 3.4]{bm},  $K\scp N^2(K)$. It follows that $K\scp \dot\sigma$ and, in particular, $K\searrow \dot\sigma$.
\end{proof}


\section{The duality in terms of reduced lattices}

\begin{definition}
A finite poset $X$ is called a \it reduced lattice \rm if every lower bounded set of $X$ has an infimum.
\end{definition}

Equivalently, a poset is a reduced lattice if and only if it is obtained from a finite lattice by deleting the maximum and the minimum.
Note that if $X$ is a reduced lattice, every upper bounded set has a supremum.  For example, the face poset $\x(K)$ of any finite simplicial complex $K$ is a reduced lattice.

\begin{definition}\label{te}
 Given a reduced lattice $X$, we denote by $m(X)$ the set of its minimal elements and by $T(X)$ the simplicial complex whose vertex set is $m(X)$ and whose simplices are the subsets of $m(X)$ which are bounded above.
\end{definition}

Note that this construction is related to the $\L$-complex defined  in \cite[Section 9.2]{b1}. In fact, $T(X)=\L(X^{op})$, the $\L$-complex of the opposite of $X$.

\begin{remark}\label{thomotopy}
It is clear that $T(\x(K))=K$ for any finite simplicial complex $K$. Moreover, by \cite[Section 9.2]{b1}, for any reduced lattice $X$, the complex $T(X)$ is homotopy equivalent to the standard order complex $\k(X)$ whose simplices are the non-empty chains of $X$.
\end{remark}


\begin{definition} Given a reduced lattice $X$ and a set $V$ such that $m(X)\subseteq V$, we define its Alexander dual as the reduced lattice $\x(T(X)^*)$. Here $T(X)^*$ denotes the Alexander dual of the simplicial complex $T(X)$ with respect to the ground set $V$.
\end{definition}

By Remark \ref{thomotopy}, the simplicial Alexander duality immediately extends to this context as follows.

\begin{proposition}
Given a reduced lattice $X$ and a set $V$ such that $m(X)\subseteq V$, then for any $i$
$$H_i(X)\cong H^{n-i-3}(X^*),$$
where $n=\# V$.
\end{proposition}

The (co)homology of a poset $X$ is the (co)homology of its associated order complex $\k(X)$. It is known that a finite poset is essentially a finite topological space and its homology groups coincide with the homology groups of the associated order complex (see \cite{b1,bm}). Therefore this result also can be used to investigate the topology of finite spaces.

\begin{remark} 
Since $K=T(\x(K))$, this version of the duality extends the simplicial version. Note also that in general $X^{**}\neq X$, unless $X=\x(K)$ for some simplicial complex $K$. In fact, $X^{**}=\x(T(X))$.
\end{remark}

\begin{example} Figure \ref{reduced} shows a reduced lattice $X$, which is not the face poset of a complex,  and its dual $X^*$.

\begin{figure}[h]
\begin{minipage}{0.5\linewidth}

  \begin{displaymath}
  \xymatrix@C=10pt{
  & & &  \bullet \ar@{-}[dl] \ar@{-}[dr]\\
  & & \bullet \ar@{-}[dl] \ar@{-}[d] \ar@{-}[dr] & & \bullet  \ar@{-}[dr] \ar@{-}[d] \ar@{-}[dl]  &  \bullet \ar@{-}[ddrrr] \ar@{-}[dllll] &	\bullet \ar@{-}[dr] \ar@{-}[dl] \\
  &  \bullet \ar@{-}[dl] \ar@{-}[dr] & \bullet \ar@{-}[dll] \ar@{-}[drr] & \bullet \ar@{-}[dl] \ar@{-}[dr] & \bullet \ar@{-}[dll] \ar@{-}[drr] & \bullet \ar@{-}[dl] \ar@{-}[dr] & &\bullet \ar@{-}[dr] \ar@{-}[dlll] \\
    \bullet a & & \bullet b & & \bullet c & & \bullet d & & \bullet e}
  \end{displaymath}  
      \hspace{100pt} $X$
\end{minipage}\noindent

\begin{minipage}{0.5\linewidth}
\vspace{9pt}
  \begin{displaymath}
  \xymatrix@C=10pt{
\bullet \ar@{-}[d] \ar@{-}[drr] & & \bullet \ar@{-}[d] \ar@{-}[drr] & & \bullet \ar@{-}[d] \ar@{-}[drr]  & & \bullet \ar@{-}[d] \ar@{-}[dllllll] \\
    \bullet a & & \bullet c & & \bullet b & & \bullet d}
  \end{displaymath}  
      \hspace{100pt} $X^*$
\end{minipage}
\caption{A reduced lattice and its dual.\label{reduced}}
\end{figure}
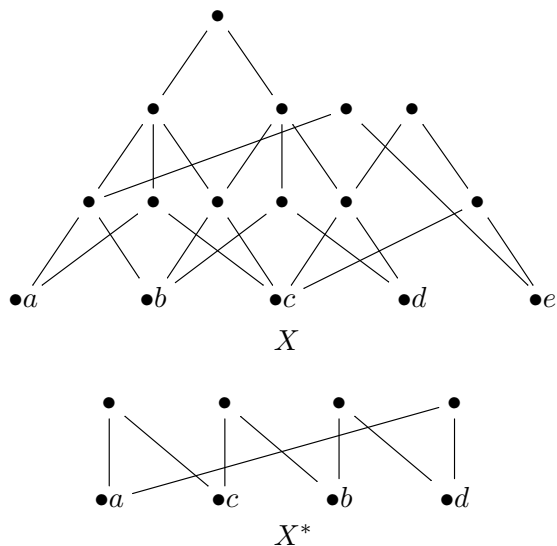

\end{example}


\end{document}